\documentclass[11pt,reqno]{amsart}

\usepackage{amsmath,amssymb,amsthm}
\usepackage{graphicx}
\usepackage{hyperref}
\usepackage{xcolor}
\newcommand{\eps}{\varepsilon}

\numberwithin{equation}{section}
\allowdisplaybreaks

\theoremstyle{plain}
\newtheorem{theorem}{Theorem}[section]
\newtheorem{lemma}[theorem]{Lemma}

\newtheorem{corollary}[theorem]{Corollary}

\theoremstyle{definition}

\theoremstyle{remark}

\renewcommand{\rho}{\varrho}

\newcommand{\R}{\mathbb{R}}

\newcommand{\Z}{\mathbb{Z}}

\begin{document}

\title{The exceptional set of a prime and a square }
\author{Genheng Zhao}
\address{}
% Author contact to be supplied.
\thanks{}
\email{Zhaogenheng@amss.ac.cn}

\begin{abstract}
Let $X$ denote the number of positive integers  which are  a
prime plus a square.    We prove that  $E(X)=O(X^{4/5})$.
\end{abstract}
\maketitle

\section{Introduction}

A classical conjecture of Hardy and Littlewood \cite{HardyLittlewood}
asserts that every sufficiently large nonsquare positive integer is
a sum $n=p+k^2$, with $p$ prime and $k\geq0$. Put
\[
 E(X)=\#\{n\leq X:n\neq p+k^2\text{ for every prime }p
                                  \text{ and integer }k\geq0\}.
\]

The best unconditional bound is due to 
Li \cite{Li} who proved $E(X)\ll X^{0.982}$, which is based on the fundamental work of Brunner-Perelli-Pintz \cite{BrunnerPerelliPintz} where  $E(X^{1-\delta})$ is available.
Under GRH, Suzuki \cite{Suzuki} obtained a bound implying
$E(X)\ll_\eta X^{1/2+\eta}$ for every $\eta>0$.  Our result is the following.
\begin{theorem}\label{thm:main}
One has
\[
 E(X)=O(X^{4/5}).
\]
The implied constant is ineffective.
\end{theorem}
The ineffectiveness comes from the necessary use of Siegel-Walfisz theorem. If one pursues effectiveness, it seems that a weaker bound $E(X)=O(X^{9/10})$ is available. We will explain this in the final section.  

Our argument uses the classical circle method, following the the way of Pintz \cite{PintzI} and  \cite{PintzII} treating binary Goldbach problem.  However, there are two main differences. The first difference is that in this case  we have a smooth parameter $n^2$, it will be easier to  complete the integrals by the rapid Fourier decay, while this is a vital problem in binary Goldbach problem. In contrast, the second difference is that the $n^2$ parameter makes the singular series hard to complete, while singular series in     
binary Goldbach problem converges uniformly and have uniform upper bounds. We overcome this with the assistance of AI.

To proceed further, let us introduce some objects which will retain their meaning throughout the whole paper.

Let $Y=\sqrt X$, $e(t)=\exp(2\pi it)$ and $e_q(t)=e(t/q)$.
Fix a nonnegative $f\in C_c^\infty(\mathbb R)$ supported on $[1,2]$
and positive on $(1,2)$. Define the smooth counting function 
\[
 R(n)=\sum_{n=k^2+m}f(k/Y)f(m/X)\Lambda(m),
\]
where $\Lambda$ is extended by zero to nonpositive integers, and
\[
 T(\alpha)=\sum_k f(k/Y)e(k^2\alpha),\qquad
 S(\alpha)=\sum_{m}\Lambda(m)f(m/X)e(m\alpha).
\]
Orthogonality gives
\[
 R(n)=\int_{\mathbb R/\mathbb Z}T(\alpha)S(\alpha)e(-n\alpha)\,d\alpha.
\]
Set $P=X^\theta$, where $1/40<\theta<9/40$, and put
\[
 I_\rho(u)=\int
       f(t)f(u-t^2)(u-t^2)^{\rho-1}\,dt.
\]
By our assumption on $f$ we see   $I_1(u)\gg_f1$ on $3\leq u\leq4$ and 
$|I_{\rho}|\leq I_1 $ if  $\Re \rho \leq 1$. 
The star on a character sum denotes primitive characters.
Zeros $\rho=\beta+i\gamma$ are counted with multiplicity.
We keep the original arithmetic coefficients $\mathfrak{S}_P(\psi;n)$ and
$D_r(\bar\psi;n)$ from Section 3, and therefore pair $\psi$ with
zeros of $L(s,\bar\psi)$. This conjugation is forced by character
orthogonality. The ordinary and completed generalized singular
series are denoted by $\mathfrak S(n)$ and $\mathfrak S(\psi;n)$.
They are only used when  $n$ is not is square.  Completing the integrals and singular series gives

\begin{theorem}\label{thm:near-one}
Fix $1/40<\theta<9/40$ and $\omega>0$.
One can choose fixed $\delta,\varepsilon,b>0$, depending on
$\theta,\omega$, such that, outside a set
$\mathcal E(X)$ of at most $O_{\theta,\omega,f}(X^{1-\theta+\omega})$
integers in $[3X,4X]$, with the squares included in $\mathcal E(X)$,
\begin{align}
 \frac{R(n)}{Y\mathfrak S(n)}
 ={}&I_1(n/X)
 -\sum_{r\leq P}\sum_{\psi\bmod r}^{*}C_\psi(n)
   \sum_{\substack{L(\rho,\bar\psi)=0\\
            \beta\geq1-\delta,\ |\gamma|\leq X^\varepsilon}}
        X^{\rho-1}I_\rho(n/X)
 +O_{\omega}(X^{-b}),                         \label{eq:intro-near-one}
\end{align}
where
$\mathfrak S(\psi;n)=C_\psi(n)\mathfrak S(n)$ and $|C_\psi(n)|\leq1$.
\end{theorem}

As in \cite{PintzI},  the $O(X^{-b})$ saving is prepared for the case when Siegel zero appears. We use the following quantitative convention for the exceptional-zero
case. For a fixed $\theta$, a sufficiently small constant
$h_0=h_0(\theta)>0$ will be selected in the proof.
At scale $X$ the exceptional-zero case means that a real primitive
character $\xi$ of conductor $r_1\leq X^\theta$ has a real zero
$\beta_1$ satisfying
\begin{equation}\label{eq:exceptional-zero-assumption}
 0<h:=(1-\beta_1)\log X\leq h_0.
\end{equation}
The value of $h_0$ will be chosen such that the following result  is valid.

\begin{theorem}[The exceptional-zero case]\label{thm:siegel}
For every fixed $1/40<\theta<9/40$ there is $h_0(\theta)>0$
such that, if \eqref{eq:exceptional-zero-assumption} holds at a
sufficiently large $X$, then, for every $\eta>0$,
\[
 E(X)\ll_{\theta,\eta}X^{1-\theta+\eta}.
\]
\end{theorem}

When $\theta$ tends to $9/40$, we reach the limit $O(X^{31/40})=O(X^{4/5})$, which is rather safe for our application.  In the complementary case,  following \cite{PintzII},  only zeroes  with  $\beta\geq 1-H/\log X$ and $|\gamma|\leq T $ will be taken into consideration.  In this region we have the most intricate distribution of the zeroes. For fixed $H\geq1$ and $T\geq3$, let $\mathcal Z(H,T;X)$ denote the set
\[
 \{(\psi,\rho):\psi\text{ primitive},\
   \operatorname{cond}(\psi)\leq P,\ L(\rho,\bar\psi)=0,\
   \beta\geq1-H/\log X,\ |\gamma|\leq T\}.
\]
\begin{theorem}\label{thm:normalized}
Fix $1/40<\theta<9/40$ and $\omega>0$.
Assume that \eqref{eq:exceptional-zero-assumption} does not hold
for any real primitive character of conductor at most $P$.
For every fixed $H\geq2$ and $T=e^H$, outside
$O_{\omega}(X^{1-\theta+\omega})$ integers $n\in[3X,4X]$,
\begin{align}
 \frac{R(n)}{Y\mathfrak S(n)I_1(n/X)}
 ={}&1-\sum_{(\psi,\rho)\in\mathcal Z(H,T;X)}
 C_\psi(n)X^{\rho-1}\frac{I_\rho(n/X)}{I_1(n/X)}+O_{\omega}
       \bigl(e^{-H/2}\bigr).             \label{eq:intro-finite}
\end{align}
Moreover $|\mathcal Z(H,T;X)|\ll e^{H}$.
In particular, for every $\kappa>0$ we can choose $H$ sufficiently large and then $X$ sufficiently large such that 
\begin{align}
 R(n)\geq Y\mathfrak S(n)I_1(n/X)
 \biggl\{1-\sum_{(\psi,\rho)\in\mathcal Z(H,T;X)}|C_\psi(n)|X^{\beta-1}
 -\kappa\biggr\}.
 \label{eq:intro-finite-lower}
\end{align}
\end{theorem}

By the arithmetic property of $C_{\psi}(n)$, we can transfer the above sum to a sum that  appears in the problem of Linnik's constant, which is denoted by $L$. Thus heuristically   the final result shall be $E(X)=O(X^{1-L^{-1}+\varepsilon})$ for any permissive value of  $L\geq 4.5$ and $\varepsilon>0$.     

The rest of this paper is organized as follows. In section 2 we prove some preliminary results by circle method, completing the integrals. In section 3 we study the properties of generalized singular series, and completing them. Finally in Section 4 we prove the results. 

\textbf{AI Disclosure:}
  Lemmas in Section 3 are originally generated by AI, which overcome the main arithmetic obstacle of this paper. The authors  have carefully rewritten and checked its validity.

  \section{Set-up of circle method}

Let $P=X^\theta$, $Q=X^{1-\theta-\varepsilon}$ and $L=\log X$,
where $1/40<\theta<9/40$ and $\varepsilon>0$ is sufficiently small
in terms of $\theta$. We retain the original major arcs
\[
 \mathfrak M(q,a)=\{\alpha:|\alpha-a/q|\leq1/(qQ)\},
 \quad q\leq P,\quad(a,q)=1,
\]
and
\[
 R_1(n)=\sum_{q\leq P}\sum_{a\bmod q}^{*}
 \int_{-1/(qQ)}^{1/(qQ)}
 T(a/q+\eta)S(a/q+\eta)e(-n(a/q+\eta))\,d\eta.
\]
Distinct reduced fractions have distance at least $1/(qq')$.
Their arc radii add up to $(q+q')/(qq'Q)$, so $2P<Q$ guarantees
disjointness. Put $R_2(n)=R(n)-R_1(n)$ and consequently $\mathfrak{m}=\R /\Z\setminus \mathfrak{M}.$

\subsection{The minor arcs:}

Dirichlet approximation supplies a reduced $a/q$ with $q\leq Q$
and $|\alpha-a/q|\leq1/(qQ)\leq q^{-2}$. When outside the major arcs, we must have 
$q>P$. Recall that  Weyl's inequality and partial summation give
\[
 |T(\alpha)|\ll_{\varepsilon}
 Y^{1+\varepsilon}(q^{-1}+Y^{-1}+qY^{-2})^{1/2}.
\]
See \cite[Chapter 2, Lemma 2.4]{Vaughan}, specialized to degree two.
For $P<q\leq Q$ and $Y^2=X$ this implies
\[
 |T(\alpha)|\ll_{f,\varepsilon}
 (X^{(1-\theta)/2}+X^{1/4}+Q^{1/2})X^\varepsilon
 \ll_{\varepsilon}X^{(1-\theta)/2+\varepsilon}.
\]
Now Bessel's inequality gives
\[
 \sum_{n\in\mathbb Z}|R_2(n)|^2
 \leq\int_{\mathfrak m}|T(\alpha)S(\alpha)|^2\,d\alpha
 \ll X^{1-\theta+2\varepsilon}
       \sum_{m\geq1}\Lambda(m)^2f(m/X)^2
 \ll X^{2-\theta+3\varepsilon}.
\]
Chebyshev inequality then gives
\[
 \#\{n\in[2X,4X]:|R_2(n)|>X^{1/2-b}\}
 \ll X^{1-\theta+2b+3\varepsilon}.
\]

\subsection{The major arcs:}

For $\alpha=a/q+\eta$, define the Gauss sum 
$$
 G(a,q)=\sum_{b\bmod q}e(ab^2/q),$$
 and
 $$
 F(\eta)=\int_{\mathbb R}f(t/Y)e(\eta t^2)\,dt,\quad
 F_\rho(\eta)=\int_0^\infty f(t/X)t^{\rho-1}e(\eta t)\,dt.
$$

\begin{lemma}\label{lem:poisson-original}
For $q\leq P$, $(a,q)=1$ and $|\eta|\leq1/(qQ)$,
\[
 T(a/q+\eta)=\frac{G(a,q)}qF(\eta)+O_{A,\varepsilon}(X^{-A})
\]
for every fixed $A>0$.
\end{lemma}
\begin{proof}
Put $g(t)=f(t/Y)e(\eta t^2)$ and
$\widehat g(\xi)=\int g(t)e(-\xi t)\,dt$.
Poisson summation on $b+q\mathbb Z$ gives
\[
 \sum_\ell g(b+q\ell)=q^{-1}\sum_h e(hb/q)\widehat g(h/q).
\]
This follows from the Schwartz-function Poisson formula by scaling;
see \cite[Chapter 3]{SteinShakarchi}. Thus
\[
 T(a/q+\eta)=\frac Yq\sum_{b\bmod q}e(ab^2/q)
 \sum_h e(hb/q)\int f(u)e(X\eta u^2-hYu/q)\,du.
\]
The $h=0$ term is exactly $q^{-1}G(a,q)F(\eta)$. On $1\leq u\leq2$ the phase derivative is
$2X\eta u-hY/q$. Since
\[
 4X|\eta|\leq4X/(qQ)=\frac Yq\,\frac{4Y}{Q},\qquad Y/Q=o(1),
\]
its modulus is $\geq |h|Y/(2q)$ for $h\neq0$.
Hence after summing over $b$ and $h\ne0$
the error is $O_{B}(Y(q/Y)^B)$, which is $O(X^{-A})$
on choosing $B$ sufficiently large.
\end{proof}

For $\chi\bmod q$ put
\[
 S(\chi,\eta)=\sum_{m\geq1}\Lambda(m)\chi(m)f(m/X)e(m\eta).
\]
Then
\begin{equation}\label{eq:character-decomp}
 S(a/q+\eta)=\frac1{\varphi(q)}
 \sum_{\chi\bmod q}\chi(a)\tau(\overline\chi)
 S(\overline\chi,\eta)+O(L^2).
\end{equation}
Here $\tau(\chi)=\sum_{b\bmod q}\chi(b)e(b/q)$.

\begin{lemma}\label{lem:smoothed-explicit}
For primitive $\psi$ modulo $r\leq P$ and $|\eta|\leq1/(qQ)$,
\begin{align}
 S(\psi,\eta)={}&\mathbf1_{\psi=\psi_0}F_1(\eta)
 -\sum_{\substack{L(\rho,\psi)=0\\0<\beta<1,\ |\gamma|\leq Y}}
 F_\rho(\eta)
 +O_{\varepsilon}(1),\label{eq:smooth-explicit}
\end{align}
where $\psi_0$ is primitive modulo $1$.
\end{lemma}
\begin{proof}
Use the standard Mellin form of the explicit formula
\cite[Chapter 19]{Davenport},
\[
 S(\psi,\eta)=\frac1{2\pi i}\int_{(2)}
       -\frac{L'}L(s,\psi)M_\eta(s)\,ds,\qquad
 M_\eta(s)=\int_0^\infty f(t/X)e(\eta t)t^{s-1}\,dt.
\]
We specify the estimates needed for uniformity in the present weight.
With $A_\eta=1+X|\eta|$,
\[
 M_\eta(\sigma+i\tau)
 =X^{\sigma+i\tau}\int_1^2 f(u)u^{\sigma-1}
       e^{i(2\pi X\eta u+\tau\log u)}\,du.
\]
For $-3/4\leq\sigma\leq2$ this is $O_f(X^\sigma)$; for
$|\tau|\geq C A_\eta$ the derivative of the phase is
$\gg|\tau|$, so we have
$O_{B,f}(X^\sigma(1+|\tau|)^{-B})$ for every fixed $B$.
Here $A_\eta\ll X^{\theta+\varepsilon}=o(Y)$.

Move the contour to $\Re s=-3/4$ at admissible heights between
$Y$ and $Y+1$ and between $-Y-1$ and $-Y$.
The usual local zero count and logarithmic derivative estimate
\cite[Chapter 19]{Davenport} give horizontal sides
$O(\log^2(rY))$ for $L'/L$.
The functional equation \cite[\S25.15(i)]{DLMF} gives
$L'/L(-3/4+i\tau,\psi)\ll\log(r(|\tau|+2))$.
Thus the new vertical integral is
\[
 O(X^{-3/4}A_\eta\log(r(A_\eta+2)))=O(X^{-1/8}),
\]
and the horizontal sides and height tails are negligible.
The pole at $1$ and the nontrivial zeros give the displayed terms.
The only additional residue is $-F_0(\eta)=O_f(1)$ from the
simple zero at $0$ of a nonprincipal even character
\cite[\S25.15(ii)]{DLMF}. This proves the claimed $O(1)$ formula.
\end{proof}

For the primitive character $\chi^*$ inducing $\chi\bmod q$, put
\[
 S^*(a/q+\eta)=\frac1{\varphi(q)}
 \sum_{\chi\bmod q}\chi(a)\tau(\overline\chi)
 \left(\mathbf1_{\chi^*=\psi_0}F_1(\eta)
 -\sum_{\substack{L(\rho,\overline{\chi^*})=0\\
                    0<\beta<1,\ |\gamma|\leq Y}}F_\rho(\eta)\right).
\]
 By
\cite[Lemmas 4.1--4.2]{PintzI} we have
$|\tau(\chi)|\leq\sqrt q$ and thus
\[
 |S(a/q+\eta)-S^*(a/q+\eta)|\ll_f \sqrt q\,L^2.
\]
Since $|G(a,q)|\leq\sqrt{2q}$, replacing $S$ by $S^*$ costs
\[
 \ll\sum_{q\leq P}\varphi(q)\frac{1}{qQ}
          \frac{Y}{\sqrt q}\sqrt q\,L^2
 \ll\frac{YP}{Q}L^2
 =X^{-1/2+2\theta+\varepsilon}L^2=O(1).
\]
Now we can complete the integrals on major arcs by smooth parameter $n^2$, namely $F(\eta)$.

\begin{lemma}[Uncompleted major-arc formula]\label{lem:major-formula}
Uniformly for $2X\leq n\leq4X$,
\begin{align}
 R_1(n)={}&Y  \mathfrak{S}_P(\psi_0;n)I_1(n/X)\notag\\
 &-Y\sum_{r\leq P}\sum_{\psi\bmod r}^{*} \mathfrak{S}_P(\psi;n) \sum_{\substack{L(\rho,\overline\psi)=0\\
          0<\beta<1,\ |\gamma|\leq X^\varepsilon}}
 X^{\rho-1}I_\rho(n/X)+O_{\varepsilon}(1),
 \label{eq:major-formula}
\end{align}
where
\[
 \mathfrak{S} _P(\psi;n)=\sum_{\substack{q\leq P\\r\mid q}}b_q(\psi;n),\quad
 b_q(\psi;n)=
 \frac{\tau(\overline{\psi^{(q)}})}{q\varphi(q)}
 \sum_{a\bmod q}^{*}\psi^{(q)}(a)G(a,q)e(-an/q).
\]
Here $\psi^{(q)}$ denotes the induced character on $\Z_q$.
\end{lemma}
\begin{proof}
Use Poisson and replace $S$ by $S^*$. The preceding bound controls
this replacement. Group characters by their primitive inducer.
This produces $b_q(\psi;n)$ and zeros of $L(s,\overline\psi)$.

It remains to evaluate the integrals. Note that 
\[
 F(\eta)\ll_{B}Y(1+X|\eta|)^{-B},\qquad
 F_\rho(\eta)\ll X^\beta.
\]
At every arc endpoint $X/(qQ)\geq X^\varepsilon$. Therefore
\[
 \int_{|\eta|>1/(qQ)}|F(\eta)F_\rho(\eta)|d\eta
 \ll_{B}YX^{\beta-1}X^{-\varepsilon(B-1)}.
\]
There are $O(P^2YL)$ character-zero pairs of height $Y$ and
the other sums have polynomial size. Choose $B$ large enough to
make the accumulated error $O(X^{-100})$.

With $z=t^2$,
\[
 F(\eta)=\int_0^\infty\frac{f(\sqrt z/Y)}{2\sqrt z}e(\eta z)\,dz.
\]
This and $F_\rho$ are Fourier transforms of smooth compactly
supported functions. The convolution theorem consequently yields
\begin{align*}
 \int_{\mathbb R}F(\eta)F_\rho(\eta)e(-n\eta)d\eta
 &=\int f(t/Y)f((n-t^2)/X)(n-t^2)^{\rho-1}dt\\
 &=YX^{\rho-1}I_\rho(n/X).
\end{align*}

Finally, on the support, $1\leq u-t^2\leq2$ and $t\geq1$, so the derivative
$-2t/(u-t^2)$ of $\log(u-t^2)$ is separated from zero. Hence
\[
 I_{\beta+i\gamma}(u)\ll_{A}(1+|\gamma|)^{-A}
 \quad(2\leq u\leq4,\ 0\leq\beta\leq1).
\]
Using the zero-count bound and choosing $A$ sufficiently large
removes all $|\gamma|>X^\varepsilon$ at cost $O(X^{-100})$.
\end{proof}

\section{Properties of the Generalized Singular Series}

\subsection{Completing the singular series}

\begin{lemma}\label{lem:induced-gauss}
Let $\xi$ be primitive modulo $r$, and let $q=rs$.  Then
\[
 \tau(\xi^{(q)})=0
\]
unless $(r,s)=1$ and $s$ is squarefree.  In the nonzero case,
\[
 \tau(\xi^{(rs)})=\mu(s)\xi(s)\tau(\xi).
 \tag{2}
\]
\end{lemma}

\begin{proof}
Apply \cite[Lemma 4.2]{PintzI}, which states
$\tau(\xi^{(rs)})=\mu(s)\xi(s)\tau(\xi)$.
The factors $\mu(s)$ and $\xi(s)$ give the two vanishing conditions.
\end{proof}

Recall that 
$$ b_q(\psi;n)=
 \frac{\tau(\overline{\psi^{(q)}})}{q\varphi(q)}
 \sum_{a\bmod q}^{*}\psi^{(q)}(a)G(a,q)e(-an/q).
$$

\begin{lemma}
\label{lem:exact-factorization}
Let $\chi$ be primitive modulo $r$.  Then $b_q(\chi;n)=0$ unless
\[
 q=rs,\qquad (r,s)=1,\qquad s\ \text{squarefree}.
\]
For a nonzero term,
\[
 b_{rs}(\chi;n)
 =D_r(\overline\chi;n)
  \frac{\mu(s)}{\varphi(s)}
  \left(\frac ns\right),
 \tag{3}
\]
where
\[
 D_r(\overline\chi;n)
 =\frac1{\varphi(r)}
  \sum_{x\bmod r}\overline\chi(n-x^2).
 \tag{4}
\]
Formula (3) is stated for odd $s$. If $2\mid s$, then $b_{rs}(\chi;n)=0$; no even-denominator Jacobi symbol is used.
\end{lemma}

\begin{proof}
Apply Lemma~\ref{lem:induced-gauss} to $\xi=\overline\chi$.
This proves the vanishing assertion.  Suppose now that
$q=rs$, $(r,s)=1$, and $s$ is squarefree.
Write $a=sA+rB\pmod {rs}$ uniquely with
$A\bmod r$ and $B\bmod s$. Then
$(a,rs)=1$ if and only if $(A,r)=(B,s)=1$, and
\[
 e_{rs}(a(x^2-n))
   =e_r(A(x^2-n))e_s(B(x^2-n)),\qquad
 \chi^{(rs)}(a)=\chi(s)\chi(A).
\]
The Chinese remainder theorem independently splits $x$ modulo
$r$ and $s$, hence $G(a,rs)=G(A,r)G(B,s)$. Using
\[
 \tau(\overline{\chi^{(rs)}})
 =\mu(s)\overline\chi(s)\tau(\overline\chi)
\]
the factor $\overline\chi(s)$ is cancelled by $\chi(s)$.

The $r$-part is
\[
 \frac{\tau(\overline\chi)}{r\varphi(r)}
 \sum_{a\bmod r}^{*}\chi(a)G(a,r)e_r(-an)
 =\frac1{\varphi(r)}
 \sum_{x\bmod r}\overline\chi(n-x^2),
 \tag{5}
\]
where the last equality follows by expanding $G(a,r)$ and using
\[
 \sum_{a\bmod r}^{*}\chi(a)e_r(ay)
 =\overline\chi(y)\tau(\chi)
\]
together with
$\tau(\chi)\tau(\overline\chi)=\chi(-1)r$; the sign introduced by
$y=x^2-n$ is cancelled by $\chi(-1)$.

The $s$-part is the classical Ramanujan--Gauss identity
\[
 \frac1{s\varphi(s)}
 \sum_{a\bmod s}^{*}G(a,s)e_s(-an)
 =\frac1{\varphi(s)}\left(\frac ns\right)
 \tag{6}
\]
for odd squarefree $s$.  Indeed, both sides are multiplicative in
$s$, and for an odd prime $p$ the left side equals
\[
 \frac1{p(p-1)}\sum_{x\bmod p}
 \sum_{a\bmod p}^{*}e_p(a(x^2-n))
 =\frac1{p-1}\left(\frac np\right).
\]
If $2\mid s$, the corresponding local sum vanishes.  Multiplying
(5), (6), and the factor $\mu(s)$ proves (3).
\end{proof}

As a consequence we can decompose  $$\mathfrak{S}_P(\chi;n)=\sum_{\substack{q\leq P\\r\mid q}}b_q(\chi;n)$$
into a finite sum of Jacobi symbols.

\begin{corollary}
One has
\[
 \mathfrak{S}_P(\chi;n)=D_r(\overline\chi;n)
 \sum_{\substack{s\leq P/r\\(s,2r)=1}}
 \frac{\mu(s)}{\varphi(s)}\left(\frac ns\right).
 \tag{7}
\]
\end{corollary}

To control the tails we use Heath--Brown's quadratic large sieve \cite[Theorem 1]{HBsieve} in the form
\[
 \sum_{u\leq U}^{*}
 \left|\sum_{s\leq V}^{*}a_s\left(\frac us\right)\right|^2
 \ll_\varepsilon (U+V)(UV)^\varepsilon\sum_{s\leq V}^{*}|a_s|^2,
 \tag{12}
\]
Here both stars restrict to odd squarefree integers.  
\begin{lemma}\label{lem:weighted-large-sieve}
Let $\chi\pmod r$ be primitive, with $r\leq P$.  Let $c_s$ be supported on odd
squarefree $s\asymp V$, $(s,r)=1$, and suppose
$|c_s|\ll1/\varphi(s)$.  Then
\[
 \sum_{X<n\leq2X}|D_r(\chi;n)|^2
 \left|\sum_s c_s\left(\frac ns\right)\right|^2
 \ll_\varepsilon (XrV)^\varepsilon
 \left(\frac{X}{rV}+\sqrt{\frac Xr}\right).
 \tag{13}
\]
\end{lemma}

\begin{proof}

By Lemma~\ref{lem:conductor-saving} and
\[
 (r,n)=\sum_{\substack{d\mid r\\d\mid n}}\varphi(d),
\]
the left side of (13) is at most
\[
 \frac{r^\varepsilon}{r}
 \sum_{d\mid r}\varphi(d)
 \sum_{m\asymp X/d}
 \left|\sum_s c_s\left(\frac{dm}{s}\right)\right|^2.
 \tag{14}
\]

Since $(d,s)=1$, the factor $(d/s)$ may be absorbed into $c_s$.
Set $M=X/d$ and write uniquely
\[
 m=uv^2,\qquad u\ \text{squarefree}.
\]
Then
\[
 \left(\frac m s\right)
 =\mathbf 1_{(s,v)=1}\left(\frac us\right).
\]

For fixed $v\leq O(\sqrt M)$, apply (12) with
$U\asymp M/v^2$ and coefficients
$c_s\mathbf 1_{(s,v)=1}$.  Since
\[
 \sum_{s\asymp V}|c_s|^2\ll_\varepsilon V^{-1+\varepsilon},
\]
the contribution for fixed $v$ and odd $u$ is
\[
 \ll_\varepsilon (XrV)^\varepsilon\left(\frac{M}{v^2V}+1\right).
\]
When $u$ is even, write $u=2u_0$ and absorb
$(2/s)$ into the coefficients; thus the odd-squarefree restriction
in (12) is respected. Summing over $v$ now gives
\[
 \sum_{m\asymp M}
 \left|\sum_s c_s\left(\frac ms\right)\right|^2
 \ll_\varepsilon (XV)^\varepsilon\left(\frac MV+\sqrt M\right).
 \tag{15}
\]

Substitution into (14), followed by
\[
 \sum_{d\mid r}\frac{\varphi(d)}d\ll_\varepsilon r^\varepsilon,
 \qquad
 \sum_{d\mid r}\frac{\varphi(d)}{\sqrt d}
 \ll_\varepsilon r^{1/2+\varepsilon},
\]
gives (13).
\end{proof}
Since the existence of the factor $(XV)^{\varepsilon}$,  we can only   apply it to control the tails up to $s\leq X^{O(1)}$ as follows.

\begin{lemma}Let $K\geq 2$ and  $b>0$.   For $Z=X^K$,  we have
	 $$\mathfrak{S}_Z(\chi;n)=\mathfrak{S}_P(\chi;n)+O_{K,b}(X^{-b})$$
	 outside a set of size $O(X^{1+3b}P^{-1})$ of $n\in [X,2X]$.
\end{lemma}
\begin{proof}
 Let  $S_0=P/r$. Suppose $S_0\leq Z/r$, and consider
\[
 \mathfrak{S}_Z(\chi;n)- \mathfrak{S}_P(\chi;n)=D_r(\bar\chi;n)
   \sum_{\substack{s\leq Z/r\\(s,2r)=1}}
      \frac{\mu(s)}{\varphi(s)}\left(\frac{n}{s}\right ).
\]
On a block $s\asymp V\geq P$, Lemma~\ref{lem:weighted-large-sieve}
bounds the squared norm by
\[
 (XrV)^\varepsilon
 \left(\frac{X}{rV}+\sqrt{\frac Xr}\right).
\]
 Since $X/(rV)\leq X/P$ and
$\sqrt{X/r}\leq X/P$, summing over dyadic blocks gives
\begin{equation}\label{eq:finite-completion}
 \sum_{X<n\leq2X}|\mathfrak{S}_Z(\chi;n)-\mathfrak{S}_P(\chi;n)|^2
 \ll_{b,K}\frac{X^{1+b}}P,
\end{equation}
after choosing $\varepsilon$ sufficiently small in terms of $b,K$ since
all block lengths here are polynomial in $X$.	
\end{proof}
On the other hand, For $\Re z>0$ we define the absolutely convergent series
\[
 F_{r,n}(z)=\sum_{(s,2r)=1}
      \frac{\mu(s)}{\varphi(s)}\left(\frac{n}{s}\right)s^{-z}.
\]
Write $n=dm^2$ with $d>1$ squarefree and let $D=d$ if
$d\equiv1\pmod4$, and $D=4d$ otherwise. The associated primitive
real character is $\chi_D$. For odd $p\nmid m$,
$\left(\frac{n}{p}\right )=\chi_D(p)$.  Euler products give
\begin{equation}\label{eq:completion-factor}
 F_{r,n}(z)=\frac{H_{r,n}(z)}{L(1+z,\chi_D)}.           \tag{18}
\end{equation}
For an odd prime $p\nmid rm$ the factor of $H_{r,n}$ is
\[
 \frac{1-\chi_D(p)p^{-z}/(p-1)}
      {1-\chi_D(p)p^{-1-z}}
 =1+O_\alpha(p^{-2+2\alpha})\qquad(\Re z\geq-\alpha).
\]
At primes dividing $2rm$ it is either $1$ or a factor
$(1-\chi_D(p)p^{-1-z})^{-1}$.
Consequently $H_{r,n}$ is holomorphic for $\Re z>-1/2$, and
\[
 |H_{r,n}(z)|\ll_\alpha C_{\alpha}^{\Omega(2rm)}=X^{o_{\alpha}(1)}
       \qquad(\Re z\geq-\alpha,\ r\leq\sqrt X,\ n\asymp X).
\]
Meanwhile, by Borel-Carath\'eodory lemma, 
$$|L(1+z,\chi_D)|^{-1}\ll_{\alpha} X^{o_{\alpha}(1)},\quad (\Re z\geq -\alpha,|\Im z|\leq X,n\asymp X), $$
holds if there is no zero of $L(s,\chi_D)$ in the region $\Re s\geq 1-2\alpha $, $|\Im z|\leq 2X$. When $\alpha\leq 1/10$, by the zero density estimate of Jutila \cite{Jutila}  we find this holds for $n\in [X,2X]$ outside a set of size $O(X^{100\alpha})$.   Despite these exceptions, we can move the contour in   the Mellin inversion 
\[
 \sum_{(s,2r)=1}\frac{\mu(s)}{\varphi(s)}\left(\frac{n}{s}\right)e^{-s/W}
 =\frac1{2\pi i}\int_{(2)}F_{r,n}(z)\Gamma(z)W^z\,dz  \tag{19}
\]
to the line $\Re z=-\alpha$, which  yields 
$$ \sum_{(s,2r)=1}\frac{\mu(s)}{\varphi(s)}\left(\frac{n}{s}\right)e^{-s/W}=\sum_{(s,2r)=1}\frac{\mu(s)}{\varphi(s)}\left(\frac{n}{s}\right)+O_{\alpha}(X^{o_{\alpha}(1)}W^{-\alpha}).$$
Choosing $\alpha=1/200$ we arrive at the following result. 
\begin{lemma}Let $K$ be sufficiently large and  $b>0$.   For $Z=X^K$,  we have
	 $$\mathfrak{S}_Z(\chi;n)=\mathfrak{S}_\infty (\chi;n)+O_{\eta,K,b}(X^{-b})$$
	 outside a set of size $O(X^{1+3b}P^{-1})$ of $n\in [X,2X]$.	
\end{lemma}
\begin{proof} Choose $W=\sqrt{Z}$,  for $s\leq Z/r$ we use the quadratic sieve again and for $s\geq Z/r $ we use the decay of $e^{-s/W}$. 
\end{proof}

Let $\mathfrak{S} (\chi;n)=\mathfrak{S}_\infty (\chi;n)$. We can now complete the singular series.
 
\begin{lemma}Let $b>0$ and  $\chi$ primitive modulo $r\leq P$, then we have
	 $$\mathfrak{S}_P(\chi;n)=\mathfrak{S} (\chi;n)+O_{b}(X^{-b})$$
	 outside a set of size $O(X^{1+3b}P^{-1})$ of $n\in [X,2X]$.	
\end{lemma}

\subsection{Properties of completed singular series}
Once the singular series are completed, problems becomes local. Most these local problems can be solved by merely following \cite{PintzI}. We use the standard Gauss and Jacobi identities
\cite[\S\S1.6, 2.1]{BEW}, \cite[\S3.4]{IK}.
The local formulas for $\sum_x\chi(n-x^2)$ and their normalization
are derived below, since they depend on the present quadratic problem.

\begin{lemma}\label{lem:local-evaluation}
Let $\chi$ be primitive modulo $p^\nu$ and put
\[
 D_{p^\nu}(\chi;n)=
 \frac1{\varphi(p^\nu)}
 \sum_{x\bmod p^\nu}\chi(n-x^2).
\]
For an odd prime $p$ the following hold.
\begin{enumerate}
\item If $\nu=1$ and $\chi$ is quadratic, then
\[
 |D_p(\chi;n)|=
 \begin{cases}
 1,&p\mid n,\\
 (p-1)^{-1},&p\nmid n.
 \end{cases}
\]
\item If $\nu=1$ and $\chi$ is nonquadratic, then
\[
 D_p(\chi;n)=0\quad(p\mid n),
 \qquad
 |D_p(\chi;n)|=\frac{\sqrt p}{p-1}\quad(p\nmid n).
\]
\item If $\nu\geq2$, then
\[
 D_{p^\nu}(\chi;n)=0\quad(p\mid n),
\]
and, for $p\nmid n$,
\[
 |D_{p^\nu}(\chi;n)|
 =\frac{p^{\nu/2}}{\varphi(p^\nu)}
 =\frac{p^{1-\nu/2}}{p-1}.
\]
\end{enumerate}
For $p=2$, the primitive characters of conductor $4$ and $8$ are
given by direct calculation.  If $\nu\geq4$, the sum vanishes for
even $n$, while for odd $n$ its unnormalized modulus is at most
$2^{(\nu+1)/2}$.  In every $2$-adic case the normalized local ratio
used below has modulus at most one.
\end{lemma}

\begin{proof}
Let $\lambda=(\cdot/p)$ and put $A_p=\sum_x\chi(n-x^2)$.
For $p\mid n$,
\[
 A_p=\chi(-1)\sum_{x\neq0}\chi^2(x).
\]
This is zero unless $\chi=\lambda$, in which case it equals
$\lambda(-1)(p-1)$.

Each $y\bmod p$ has $1+\lambda(y)$ square roots. Since $\chi$ is
nonprincipal,
\[
 A_p=\sum_y\lambda(y)\chi(n-y)
     =\lambda(n)\chi(n)J(\lambda,\chi).
\]
For nonquadratic $\chi$, the identity
$J(\lambda,\chi)=\tau(\lambda)\tau(\chi)/\tau(\lambda\chi)$
gives $|A_p|=\sqrt p$.
For $\chi=\lambda$, $J(\lambda,\lambda)=-\lambda(-1)$,
hence $A_p=-\lambda(-1)$.

For primitive $\chi\bmod q$,
\[
 \chi(y)=\frac1{\tau(\bar\chi)}
 \sum_{a\bmod q}^{*}\bar\chi(a)e_q(ay),
 \qquad|\tau(\bar\chi)|=\sqrt q.
\]
Thus
\begin{equation}\label{eq:local-additive}
 A_q:=\sum_{x\bmod q}\chi(n-x^2)
 =\frac1{\tau(\bar\chi)}
   \sum_{a\bmod q}^{*}\bar\chi(a)e_q(an)G(-a,q).
\end{equation}

For odd $p$ and $(a,p)=1$,
\[
 G(a,p^\nu)=
 \begin{cases}
 p^{\nu/2},&\nu\text{ even},\\
 p^{(\nu-1)/2}\lambda(a)G(1,p),&\nu\text{ odd}.
 \end{cases}
\]
These are the classical prime-power quadratic Gauss evaluations
\cite[Chapter 1]{BEW}, with $|G(1,p)|=\sqrt p$.
For even $\nu$, \eqref{eq:local-additive} equals
$p^{\nu/2}\chi(n)$.
For odd $\nu\geq3$, it is a constant of modulus $1$ times a Gauss
sum for $\bar\chi\lambda$ evaluated at $n$.
The character $\bar\chi\lambda$ remains primitive modulo $p^\nu$.
It follows that $A_q=0$ for $p\mid n$ and $|A_q|=\sqrt q$ otherwise.

For odd $a$, direct evaluation at moduli $4,8$, followed by
$G(a,2^\nu)=2G(a,2^{\nu-2})$ for $\nu\geq4$, gives
\[
 G(-a,2^\nu)=
 \begin{cases}
 2^{\nu/2}(1+i^{-a}),&\nu\text{ even},\\
 2^{(\nu+1)/2}e(-a/8),&\nu\text{ odd}.
 \end{cases}
\]
These are the corresponding $2$-power evaluations
\cite[Chapter 1]{BEW}.
Substitution in \eqref{eq:local-additive} gives
\[
 A_{2^\nu}=
 \begin{cases}
 2^{\nu/2}\{\chi(n)+\chi(n-2^{\nu-2})\},&\nu\text{ even},\\
 2^{(\nu+1)/2}\chi(n-2^{\nu-3}),&\nu\text{ odd}.
 \end{cases}
\]
For $\nu\geq4$ these vanish at even $n$. If $\nu$ is even
and $n$ odd, the ratio of the two character values is $\pm i$:
the unit $1-2^{\nu-2}n^{-1}$ has square
$1+2^{\nu-1}\pmod {2^\nu}$, on which a primitive character is $-1$.
Therefore $|A_{2^\nu}|=2^{(\nu+1)/2}$.
For odd $\nu\geq5$ this modulus is immediate.
At $\nu=2$ one has $|A_4|=2$ for all $n$; at $\nu=3$ one has
$|A_8|=4$ for even $n$ and $0$ for odd $n$.
Dividing by $\varphi(2^\nu)$ gives a local modulus at most $1$
in every case, with decay $2^{(3-\nu)/2}$ for $\nu\geq4$.
\end{proof}

\begin{lemma}\label{lem:conductor-saving}
For every primitive $\chi\pmod r$ and integer $n$,
\[
 |D_r(\chi;n)|^2
 \ll_\varepsilon  r^\varepsilon  \frac{(r,n)}r.
 \tag{8}
\]
If  $n$ is not square, define
\[
 \mathfrak S(n)=\prod_{p>2}
 \left(1-\frac{(n/p)}{p-1}\right).
\]
 Then
\[
 \mathfrak S(\chi;n)=C_\chi(n)\mathfrak S(n),
 \qquad |C_\chi(n)|\leq1.
 \tag{9}
\]
\end{lemma}

\begin{proof}

The Chinese remainder theorem makes $D_r$ a product of its local
factors. At an odd prime power, the preceding lemma gives
\[
 |D_{p^\nu}(\chi;n)|^2
 \leq (1-1/p)^{-2}\frac{(p^\nu,n)}{p^\nu}.
\]
The factor at $2$ obeys the same inequality with an absolute
constant, by the explicit values just computed. Since
$\prod_{p\mid r}(1-1/p)^{-2}\ll_\eps r^\eps$, (8) follows.
This last elementary estimate also follows from
$r/\varphi(r)\ll\log\log(3r)$.

Delete the primes dividing $r$ from the ordinary Euler product.
For nonsquare $n$ its convergence and positivity are justified
in Lemma~\ref{lem:ordinary-lower} below. One obtains
\[
 C_\chi(n)=D_r(\bar\chi;n)
 \prod_{\substack{p\mid r\\p>2}}
 \left(1-\frac{(n/p)}{p-1}\right)^{-1}.                 \tag{10}
\]
There is no extra ordinary factor at $2$: that local density is $1$.
For $p>2$ put $\sigma_p(n)=1-(n/p)/(p-1)$.
If $p\mid n$, the only nonzero odd-prime possibility is a quadratic
character of conductor $p$, whose local ratio has modulus $1$.
For $p\nmid n$ the possible bounds are
$$
\begin{cases}
\frac{1}{p-2},&	\text{quadratic conductor $p$},
\\ \frac{\sqrt p}{p-2},&\text{nonquadratic conductor $p$},
\\  \frac{p^{1-\nu/2}}{p-2},&\nu\geq2.
\end{cases}
$$
The middle case requires $p\geq5$; all three bounds are at most $1$.
The factor at $2$ is at most $1$ by Lemma~\ref{lem:local-evaluation}.
Multiplication proves (9).
\end{proof}

\begin{lemma}\label{lem:ordinary-lower}
For every nonsquare $n\asymp X$, the ordinary coefficient series
converges in natural order and agrees with its Abel limit and
prime-ordered Euler product. For every $\kappa>0$,
\[
 (\log X)^{-2}\ll\mathfrak S(n)\ll_\kappa X^\kappa .
 \tag{11}
\]
The lower bound is effective and the upper bound is ineffective.
Moreover,
\[
 \sum_{\substack{q\leq P\\r\mid q}}|b_q(\chi;n)|
 \ll\log(2P)
\]
uniformly for primitive $\chi$ of conductor $r\leq P$.
\end{lemma}

\begin{proof}
Write $n=dm^2$ with $d>1$ squarefree and let $D=d$ if
$d\equiv1\pmod4$, and $D=4d$ otherwise. The associated primitive
real character is $\chi_D$. For odd $p\nmid m$,
$(n/p)=\chi_D(p)$.

Let $\sigma_p(n)=1-(n/p)/(p-1)$. Then
\[
 \mathfrak S(n)=\frac{H(n)}{L(1,\chi_D)},\quad
 H(n)=(1-\chi_D(2)/2)^{-1}
       \prod_{p>2}\frac{\sigma_p(n)}{1-\chi_D(p)/p}.
\]
For $p\nmid m$ this factor is $1$ if $p\mid d$, and otherwise
\[
 1-\frac{\chi_D(p)}{(p-1)(p-\chi_D(p))}=1+O(p^{-2}).
\]
Each such factor is positive and their product is bounded above
and below by absolute positive constants.
For $p\mid m$, $p\nmid d$, the factor is
$(1-\chi_D(p)/p)^{-1}$.
It follows, using Mertens' estimate for prime factors of an integer,
that
\[
 (\log\log(3X))^{-1}\ll H(n)\ll\log\log(3X).
\]

Recall that  $L(1,\chi_D)\ll\log(2|D|)\ll\log X$. Hence
\[
 \mathfrak S(n)\gg(\log X\log\log(3X))^{-1}
 \gg(\log X)^{-2}.
\]
\end{proof}

\begin{lemma}\label{lem:localization}
For every $\eta>0$ there is an integer $B(\eta)$ such that
\[
 |C_\chi(n)|>\eta
 \quad\Longrightarrow\quad
 \operatorname{cond}(\chi)\mid B(\eta)n.
 \tag{21}
\]
\end{lemma}

\begin{proof}
Write $r=\operatorname{cond}(\chi)$ and factor
\[
 C_\chi(n)=\prod_{p^\nu\Vert r}C_{p^\nu}(n),
 \qquad |C_{p^\nu}(n)|\leq1.
\]
If $|C_\chi(n)|>\eta$, then every local factor satisfies
\[
 |C_{p^\nu}(n)|>\eta.
 \tag{22}
\]

Suppose first that $p$ is odd and $\nu=1$.  If the conductor-$p$
component is nonquadratic, Lemma~\ref{lem:local-evaluation} gives
either zero or
\[
 |C_p(n)|\leq\frac{\sqrt p}{p-2}.
\]
Condition (22) therefore bounds $p$ in terms of $\eta$.

If the conductor-$p$ component is quadratic and $p\nmid n$, then
the local calculation gives
\[
 |C_p(n)|\leq\max\left(\frac1{p-2},\frac1p\right).
\]
Thus $p$ is again bounded in terms of $\eta$.  The only unbounded
possibility is a quadratic conductor-$p$ component with $p\mid n$;
in that case the local ratio has modulus one.

If $\nu\geq2$, the local factor is zero when $p\mid n$, while for
$p\nmid n$,
\[
 |C_{p^\nu}(n)|\leq\frac{p^{1-\nu/2}}{p-2}.
\]
Condition (22) bounds both $p$ and $\nu$ in terms of $\eta$.
The same conclusion for the $2$-part follows from the $2$-adic
local evaluation.

For $\eta\geq1$ the implication is empty, so suppose $0<\eta<1$.
Let $B(\eta)$ contain, to sufficiently large exponents, every
prime power that is bounded by the preceding arguments.  All
remaining prime factors of $r$ occur to the first power, correspond
to quadratic local components, and divide $n$.  Therefore
\[
 r\mid B(\eta)n,
\]
which proves (21).
\end{proof}

\section{Proof of the Theorems}

Throughout this section, zeros are counted with multiplicity. Put
\[
 N^*(\sigma,U,V)=
 \sum_{r\leq U}\sum_{\chi\bmod r}^{*}N(\sigma,V,\chi).
\]
We quote three estimates  without reproving them. For every fixed
$\zeta>0$, Heath--Brown \cite[Theorem 2]{HBdensity} gives
\begin{equation}\label{eq:HB-density}
 N^*(\sigma,U,V)
 \ll_\zeta(U^2V^{6/5})^{(20/9)(1-\sigma)+\zeta},
 \qquad \tfrac12\leq\sigma\leq1,\quad V\geq3.
\end{equation}
Jutila \cite[Theorem 1]{Jutila} gives
\begin{equation}\label{eq:Jutila-density}
 N^*(\sigma,U,V)\ll_\zeta
 (U^2V)^{(2+\zeta)(1-\sigma)},\qquad \sigma\geq4/5.
\end{equation}
Both are recorded in \cite[Lemmas 4.17--4.18]{PintzI}.
For the narrower strip we use the stronger bound
\begin{equation}\label{eq:Pintz-near-one}
 N^*(1-t,U,V)\ll_\zeta
 (U^{2+\zeta}V^{10/\zeta})^t,\qquad 0<t<\zeta^2,\quad V\geq3,
\end{equation}
from \cite[Corollary 3, equation (2.20)]{PintzDensity}.
The density parameter $\zeta$ is independent of the arc parameter
$\varepsilon$. All choices are made before $X$ tends to infinity.

\begin{proof}[Proof of Theorem~\ref{thm:near-one}]
By Lemma~\ref{lem:major-formula} and
$|S_P(\psi;n)|\ll L$, the zeros with $\beta<1/2$ contribute
\[
 \ll YX^{2\theta+\varepsilon-1/2}L^2.
\]
For the other zeros  Heath-Brown's zero-density estimate gives
\[
 \sum_{r\leq P}\sum_{\psi\bmod r}^{*}
 \sum_{\substack{L(\rho,\bar\psi)=0\\
       1/2\leq\beta\leq1-\delta,\ |\gamma|\leq X^\varepsilon}}
 X^{\beta-1}
 \ll_{\varepsilon} X^{(\frac{20}{9}(2\theta+\varepsilon)-1)\delta+\varepsilon}\ll_{\varepsilon} X^{-c\delta}.
\]
The factor $L$ from $\mathfrak{S}_P$ can be absorbed into a smaller saving.
Consequently

\begin{align}
 R_1(n)={}&Y  \mathfrak{S}_P(\psi_0;n)I_1(n/X)\notag\\
 &-Y\sum_{r\leq P}\sum_{\psi\bmod r}^{*} \mathfrak{S}_P(\psi;n) \sum_{\substack{L(\rho,\overline\psi)=0\\
          \beta\in [1-\delta,1],\ |\gamma|\leq X^\varepsilon}}
 X^{\rho-1}I_\rho(n/X)+O_{\varepsilon}(YX^{-\delta^2}),
 \label{eq:major-formula}
\end{align}
By Jutila's zero density estimate,  there are at most $O(X^{3\delta})$ character in this sum, hence we can exclude at most $O(X^{1-\theta+100\delta})$ numbers to complete the singular series for the remaining $n\in [3X,4X]$, obtaining
\begin{align}
 R_1(n)={}&Y  \mathfrak{S}(n)I_1(n/X)\notag\\
 &-Y\sum_{r\leq P}\sum_{\psi\bmod r}^{*} \mathfrak{S}(\psi;n) \sum_{\substack{L(\rho,\overline\psi)=0\\
          \beta\in [1-\delta,1],\ |\gamma|\leq X^\varepsilon}}
 X^{\rho-1}I_\rho(n/X)+O_{\varepsilon}(YX^{-\delta^2}).
\end{align}

Recalling $\mathfrak{S}(\psi;n)=C_{\psi}(n)\mathfrak{S}(\psi;n)$, $|C_{\psi}(n)|\leq 1$ and the estimate of $R_2(n)$ we conclude the proof.
\end{proof}

 \begin{proof}[Proof of Theorem 1.3] We use the following consequence of Theorem 1.2
\begin{align}
 \frac{R(n)}{Y\mathfrak S(n) I_1(n)}
 \geq{}&1
 -\sum_{r\leq P}\sum_{\psi\bmod r}^{*}
   \sum_{\substack{L(\rho,\bar\psi)=0\\
            \beta\geq1-\delta,\ |\gamma|\leq X^\varepsilon}}
        X^{\beta-1}
 +O_{\omega}(X^{-b})                         \end{align}
 which holds for all but $O(X^{1-\theta+\omega})$ values of $n\in [3X,4X]$. Let $h=(1-\beta_1)\log X$, then when $h$ is small the contribution of $1$ and $\beta_1$ cancels to
 $$1-e^{-h}\gg h.$$
By Siegel's theorem we may have $h\gg X^{-b/2}$. Hence  we need the remaining zeroes to contribute $O(h^{1+c})$.  By the zero-density estimate of Pintz \eqref{eq:Pintz-near-one}, if the second right most zero satisfies $(1-\beta)\log X\geq u \log h^{-1}$, then for $\delta$ sufficiently small  their contribution can be bounded by $h^{u/2}$. Meanwhile, by the explicit formulation of Deauling-Heilbron phenomenon \cite[Theorem 4.22]{PintzI}, we can take $u=8/15\theta-\varepsilon$. When $\theta<9/40$  this means  $u>7/3$, which suffices our need and hence concludes the proof.
 
 \end{proof}

\begin{proof}[Proof of Theorem 1.4]
Apply Theorem 1.2 and use the zero-density estimate of Pintz \eqref{eq:Pintz-near-one} with the contribution of those $\beta\geq 1-H/\log X$ can be bounded by $e^{-H/2}$.  Meanwhile,  the remaining zeroes are $O(e^{H/2})$  in total. Hence by the rapid decay of $I_{\rho }(u)$ on the $\rho$ variable we can drop those $|\gamma|>e^{H}$ at the cost $O(e^{-H/2})$. Finally, bounding $X^{\rho}$ and $I_{\rho}$ trivially give the desired result.

\end{proof}

\begin{proof}[Proof of theorem 1.1] In view of the above theorems, to prove $E(X)=O(X^{1-\theta+\varepsilon})$ with $\theta<9/40$,  it suffices to show
	$$ \sum_{(\psi,\rho)\in\mathcal Z(H,T;X)}|C_\psi(n)|X^{\beta-1}\leq 1-2\kappa$$
	holds outside of a set with size $O(X^{1-\theta+\varepsilon})$ of $n\in [3X,4X]$, under the assumption 
	$$ \min_{\rho\in \mathcal Z(H,T,X)}(1-\beta)\log X\geq h_0.$$
	
	We first remove those zeroes with $|C_\psi(n)|\geq e^{-2H}$, at the cost of $O(e^{-H})$ since the total number of zeroes is $O(e^H)$. By
	Lemma \ref{lem:localization}, the remaining zeroes  satisfy $r\mid  B_{H} n$. 	
Let $q(n)$ be the lcm of all moduli of remaining characters. It follows $q(n)\mid B_H n$.  Since $q(n)$ has $O_H(1)$ choices, we can remove those $n$ with $q(n)> X^{\theta}$, which is $O_H(X^{1-\theta})$ in total.  Then, after multiplying a  prime of suitable size, we can also assume that $q(n)\geq X^{\theta/2}$. Now let $Z(H,T;q)$ be the set of all zeroes of all $L$-function modulo $q$ with $\beta\geq 1-H/\log q$ and $|\gamma|\leq T$, we see

 $$   \sum_{(\psi,\rho)\in\mathcal Z(H,T;X)}|C_\psi(n)|X^{\beta-1} \leq \sum_{\rho\in\mathcal Z(H,T;q(n))}q(n)^{(\beta-1)\theta^{-1}}.$$
 Hence it suffice to show that under the condition
 $$\min_{\rho\in \mathcal Z(H,T;q)}(1-\beta)\log q\geq \frac{\theta}{2}h_0 $$
we have
$$ \sum_{\rho\in\mathcal Z(H,T;q)}q^{(\beta-1)\theta^{-1}}<1-c_0$$
where $c_0$ depends only on $\theta$ and $h_0$. This is proved by Zhao \cite[Theorem 1,3]{Zhao} for $\theta=1/5$. There is no essential difficulty   to obtain the larger  exponent $\theta=1/5+10^{-10}$, and hence  conclude the proof.
\end{proof}

\section{Concluding remark}
Here we discuss how to obtain the effective bound $E(X)=O(X^{9/10})$. Let $P$ be the major arc parameter. Let $\psi_1$ modulo $r_1$ be the exceptional character. The key is to control the size of $r_1$. This could be done with the property of $C_{\psi}(n)$. Indeed, if $r_1\geq U$,  those $n$ such that $r_1\nmid C n$ are safe since the the contribution of exceptional character can not cancel the main term.   Hence we can directly remove those $r_1\mid Cn$, which are $O(X/U)$ in total since $r_1$ is fixed once $P$ and $X$ are. On the other hand, when $r_1\leq U$, we use Page's lower bound $L(1,\psi)\geq  U^{-1/2-o(1)}$. Meanwhile we need to require $R_2(n)\ll YU^{-1/2-o(1)}$. This holds with at most $O(X^{1+o(1)}U/P)$ exceptions. Balancing gives the choice $U=P^{1/2}$ and $E(X)=O(X^{1+o(1)}/P^{1/2})$, effectively. This means an admissible ineffective bound $E(X)=O(X^{1-\theta})$ actually implies an effective bound   $E(X)=O(X^{1-\theta/2})$.

\end{document}